\documentclass[12pt, reqno]{amsart}

\usepackage{color}
\usepackage{graphicx}
\usepackage[english]{babel}
\usepackage{times}
\usepackage{mathtools}
\usepackage{amsmath}
\usepackage{amsfonts}
\usepackage{amssymb}
\usepackage{amsthm}
\usepackage{transparent}
\usepackage{array}
\usepackage{listings}
\usepackage{enumitem}

\usepackage[margin=1.1in]{geometry}

\usepackage{hyperref}

\hypersetup{
    colorlinks=true,
    citecolor=blue,
    linkcolor=blue,
    filecolor=magenta,      
    urlcolor=cyan,
    pdftitle={Overleaf Example},
}

\def\@fnsymbol#1{\ensuremath{\ifcase#1\or *\or \dagger\or \ddagger\or
   \mathsection\or \mathparagraph\or \|\or **\or \dagger\dagger
   \or \ddagger\ddagger \else\@ctrerr\fi}}

\def\eps{\varepsilon}
\def\S{\mathcal{S}}
\def\H{\mathcal{H}}
\def\M{\mathcal{M}}
\def\R{\mathbb{R}}
\def\Z{\mathbb{Z}}
\def\N{\mathbb{N}}
\def\C{\mathbb{C}}

\def\loc{\mathrm{loc}}

\theoremstyle{plain}
\newtheorem{thm}{Theorem}[section]
\newtheorem{cor}[thm]{Corollary}
\newtheorem{lemma}[thm]{Lemma}
\newtheorem{prop}[thm]{Proposition}

\theoremstyle{definition}
\newtheorem{defo}[thm]{Definition}
\newtheorem{ex}[thm]{Example}
\newtheorem{rem}[thm]{Remark}

\title[Boundedness of the segment multiplier on r.i. spaces]{Characterization of the boundedness of the segment multiplier on rearrangement-invariant spaces}

\author[M. F. Barea-Fern\'andez]{Miguel F. Barea-Fern\'andez}
\address{Miguel F. Barea-Fern\'andez, Departamento de An\'alisis Matem\'atico y Matem\'atica Aplicada, Fa\-cul\-tad de Matem\'aticas, Universidad Complutense de Madrid, Plaza de Ciencias 3, 28040 Madrid, Spain.
\texttt{https://orcid.org/0009-0005-5799-5023}}
\curraddr{}
\email{mibarea@ucm.es}
\thanks{The first author was partially supported by grant PID2024-155917NB-I00, funded by MCIN/AEI/10.13039/501100011033, and by an FPU Grant FPU23/00891, from Ministerio de Ciencia, Innovaci\'on y Universidades (Spain).}

\author[R. Kerman]{Ron Kerman}
\address{Ron Kerman, Department of Mathematics, Brock University, St. Catharines, ON L2S 3A1, Canada.
\texttt{https://orcid.org/0000-0002-6883-7793}}
\curraddr{}
\email{kermanronald@gmail.com}
\thanks{}

\author[J. Lang]{Jan Lang}
\address{Jan Lang, Department of Mathematics, The Ohio State University, Columbus, OH, United States, Department of Mathematics, Faculty of Electrical Engineering, Czech Technical University in Prague, Czech Republic. \texttt{https://orcid.org/0000-0003-1582-7273}}
\curraddr{}
\email{lang@math.osu.edu}
\thanks{}

\author[J. Soria]{Javier Soria}
\address{Javier Soria, Departamento de An\'alisis Matem\'atico y Matem\'atica Aplicada, Fa\-cul\-tad de Matem\'aticas, Universidad Complutense de Madrid, Plaza de Ciencias 3, 28040 Madrid, Spain and ICMAT. \texttt{https://orcid.org/0000-0003-3098-7056}}
\curraddr{}
\email{javier.soria@ucm.es}
\thanks{The fourth author was partially supported by grants PID2024-155917NB-I00 and CEX-2023-001347-S, funded by MCIN/AEI/10.13039/501100011033, and Grupo UCM-970966.}


\date{}

\begin{document}

\begin{abstract}
    Given rearrangement-invariant (r.i.) spaces $X$ and $Y$, we show that the segment multiplier, the truncated Hilbert transform, and the discrete Hilbert transform (on the associated discretized space) are bounded from $X$ to $Y$ simultaneously. Moreover, when $X = Y$ this boundedness is characterized by a condition on a pair of Boyd-type indices. We also exhibit an r.i.\ space on which the segment multiplier is bounded while the (non-truncated) Hilbert transform fails to be bounded.

\end{abstract}

\maketitle

\section{Introduction}

Fourier multipliers are a broad class of operators of great interest in harmonic analysis, and they include classical objects such as the Hilbert transform and convolutions. In general, given a bounded function $m\in L^\infty(\R^n)$, the multiplier operator $T$ associated with the symbol $m$ is defined for a Schwartz function $f\in \S(\R^n)$ by the equality
$$ \widehat{Tf}(\xi) = m(\xi) \hat f(\xi), $$
where $\hat f(\xi) = \int_{\R^n} f(x) e^{-2\pi i x \xi} \, dx$ denotes the Fourier transform of $f$ (observe that $T$ extends to a bounded operator on $L^2(\R^n)$ by Plancherel's theorem, with $\|T\|_{L^2\to L^2} \leq \|m\|_\infty$). The multiplier norm of $m$ is defined as the operator norm of $T$. Because the Fourier transform maps convolutions to pointwise products of functions and vice versa, the operator $T$ can also be expressed as
$$ Tf = \check{m} * f, $$
where $\check m$ denotes the inverse Fourier transform of $m$ in the distributional sense. For an introduction to multipliers, see \cite[Section~2.5]{grafakos}.

The fundamental groundwork for $L^p$-bounded multipliers was established by H\"ormander \cite{hormander}, with some important results found by Mihlin \cite{mihlin} and Marcinkiewicz \cite{marcin}, among others. However, even in the $L^p(\R^n)$ setting, most multiplier classes remain far from being completely understood. Thus, sometimes the work is focused on studying the boundedness of particular examples of multipliers such as, notably, the ball multiplier for dimension $n\geq 2$, given by the symbol $m = \chi_B$, the characteristic function of the Euclidean ball in $\R^n$. If $n = 1$, the equivalent operator is the segment multiplier, which had already been known to be bounded in $L^p(\R)$ when $1 < p < \infty$ since the work done by M. Riesz on the Hilbert transform \cite{riesz}. However, this fact could not be translated in the case of $\chi_B$ to any dimension $n \geq 2$, and it came as a surprise when Fefferman proved that the ball multiplier is bounded on $L^p(\R^n)$ only when $p = 2$ \cite{fefferman}. Due to the simplicity of these symbols, they serve as powerful examples when trying to gain information on multiplier classes as a whole.

Let us look at the one-dimensional case: define the segment multiplier and the Hilbert transform, respectively, as
\begin{equation*}
Sf(x):= \int_{-\infty}^{\infty} \frac{\sin(x-y)}{x-y} f(y) \, dy, \qquad Hf(x):=\text{p.v.} \int_{-\infty}^{\infty} \frac{f(y)}{x-y} \, dy.
\end{equation*}
With these definitions, their multiplier symbols are \vspace{2pt}
$$ m_S(\xi) = \pi \chi_{[-\frac{1}{2\pi}, \frac{1}{2\pi}]}(\xi), \qquad m_H(\xi) = -i\pi \text{sgn}(\xi). $$ \vspace{2pt}
A direct computation yields
$$ m_S(\xi) = i \, \frac{m_H(\xi+\frac{1}{2\pi}) - m_H(\xi - \frac{1}{2\pi})}{2}, $$
so it follows that if $H$ is bounded, $S$ will be as well. Furthermore, the range of exponents $p$ for which each operator is bounded on $L^p$ is the same, as both are unbounded for $p = 1, \infty$ (the relation of their norms in this context has been studied in \cite{wu}). This raises the question of whether this simultaneous boundedness property persists beyond the scale of $L^p$ spaces.

We consider the broader class of function spaces called rearrangement-invariant (r.i.) spaces, which includes every $L^p$ as well as other important classes, such as Orlicz spaces and Lorentz spaces. In the r.i.\ setting, some of the foundational operators in harmonic analysis have been studied extensively, including the Fourier transform (see \cite{brancol} and \cite{boso}), the Hardy-Littlewood maximal operator (see \cite{lorentz} and \cite{shimogaki}) and, remarkably, the Hilbert transform, for which Boyd~\cite{bo_hilbert} introduced a pair of indices related to each r.i.\ space that characterize its boundedness among several other operators, like the maximal operator. The aim of this paper is to find a suitable description of the boundedness of the segment multiplier on r.i.\ spaces, a context where it has not been studied yet.

In this direction, Theorem \ref{cor:mainri} provides a complete characterization, which also extends to the boundedness of the truncated Hilbert transform and the discrete Hilbert transform (on an appropriate sequence space). This result for the segment multiplier is analogous to a characterization on weighted $L^p$ spaces, established in a parallel unpublished work by some of the authors. Additionally, Example \ref{ex:l1+lp} gives simple examples of r.i.\ spaces where the segment multiplier is bounded but the Hilbert transform is not, thanks to the use of Theorem \ref{cor:mainri}. All of these results complement \cite{fefferman} and \cite{wu} by expanding the knowledge in dimension $n = 1$.

The organization of the article is as follows: in Section \ref{sec:prel} we state the definitions and properties needed in the rest of this work. In Section~\ref{sec:bfs} we prove Theorem \ref{charbfs}, the equivalence between the boundedness of the segment multiplier and the truncated Hilbert transform on Banach function spaces under certain conditions. In Section \ref{sec:ri} we prove our main result, Theorem \ref{mainri}, for which we have to define the auxiliary discretized space $\ell_X$ (see Proposition \ref{prop:ellx}); finally, Theorem \ref{cor:mainri} provides the aforementioned characterization and, in Example \ref{ex:l1+lp}, we construct elementary r.i.\ spaces where the boundedness of the segment multiplier and the Hilbert transform are not equivalent, in contrast to what happens on $L^p(\R)$.

\section{Preliminary definitions and results} \label{sec:prel}

If $(R,\mu)$ is a measure space and $E \subseteq R$ is a subset, we will denote by $\chi_E$ the characteristic function of the set $E$. If $\mu$ is the Lebesgue measure over $R = \R^n$, we will denote the measure of $E$ as $|E|$.

Throughout this paper, we will write $ f \lesssim g $ to mean an inequality of the form $ f \leq C g $ where $C$ is a constant independent of any relevant variables. Similarly, the notation $ f(s) \approx g(s) $ is equivalent to having both $ f(s) \lesssim g(s) $ and $ g(s) \lesssim f(s) $.

\subsection{Banach function spaces}

In this section we introduce the main spaces that we will be working with, as well as their basic properties. Many of these definitions and statements are taken from \cite[Chapter I]{bensharp}. If $(R,\mu)$ is a $\sigma$-finite measure space, denote by $\M(R)$ the set of $\mu$-measurable functions $f:R \to \R$ (or $\C$) and $\M^+(R)$ the subset of non-negative functions in $\M(R)$.

\begin{defo}
    Let $\rho: \M^+(R) \to [0,\infty]$ be a functional defined on all non-negative measurable functions. We say that $\rho$ is a Banach function norm if it satisfies the following properties:
    \begin{enumerate}[label=(\alph*)] 
        \item Norm properties:
        \begin{itemize}
            \item If $f\in \M^+$, then $\rho(f) = 0$ if and only if $f = 0$ a.e.
            \item For every $f\in \M^+$ and $a \geq 0$, $\rho(af) = a \, \rho(f)$.
            \item For every $f,g \in \M^+$, $\rho(f+g) \leq \rho(f) + \rho(g)$.
        \end{itemize}
        \item Monotonicity: if $f,g\in\M^+$ with $f\leq g$ a.e., then $\rho(f) \leq \rho(g)$.
        \item Fatou property: if $\{f_n\}_n\subseteq \M^+$ and $f_n \nearrow f$ a.e., then $\rho(f_n) \nearrow \rho(f)$.
        \item Simple functions: if $E\subseteq R$ with $\mu(E) < \infty$, then $\rho(\chi_E) < \infty$.
        \item H\"older inequality for characteristic functions: if $E\subseteq R$ with $\mu(E) < \infty$, then there exists a constant $C_E>0$ such that for every $f\in\M^+$ the following holds:
        $$ \int_E f(x)\,d\mu(x) \leq C_E \rho(f). $$
    \end{enumerate}
\end{defo}

\begin{defo}
    Let $\rho$ be a Banach function norm. The normed vector space defined as
    $$X = X(R) = \{f\in\M(R) \, : \, \rho(|f|) < \infty \} $$
    is called a Banach function space (shortened to BFS), and the norm of each $f\in X$ is
    $$\|f\|_X := \rho(|f |).$$
\end{defo}

In the context of Banach function spaces there is the notion of associate space (or K\"othe dual space), which behaves as the classical dual space while still remaining in the BFS class.

\begin{defo}
    Let $\rho$ be a Banach function norm associated to the BFS $X(R)$, and define for $f\in\M^+(R)$ the functional
    $$\rho'(f) = \sup\Big\{ \int_R f(x)g(x)\, d\mu(x) \, : \, g\in\M^+, \; \rho(g) \leq 1 \Big\}. $$
    Then, $\rho'$ is a Banach function norm called the associate norm to $\rho$, and the associate space is
    $$X' = \{f\in\M \, : \, \rho'(|f|) < \infty \}.$$
\end{defo}

\subsection{Rearrangement-invariant spaces}

In this section, we assume that $(R,\mu)$ is either a non-atomic measure space, or a completely atomic space with all of its atoms having equal measure. The standard examples of these spaces are $\R$ with the Lebesgue measure and $\Z$ with the counting measure, respectively. The following definitions and properties come from \cite[Chapters II and III]{bensharp}.

\begin{defo}
    Let $f\in \M(R)$. Its distribution function $\lambda_f: [0,\infty) \to [0,\infty]$ is defined as
    $$\lambda_f(t) = \mu(\{ x\in R \, : \, |f(x)| > t \}), \qquad t\geq 0. $$

    The decreasing rearrangement of $f$ is the function $f^*:[0,\infty) \to [0,\infty]$ defined as
    $$f^*(t) = \inf\{ s>0 \, : \, \lambda_f(s) \leq t \}, \qquad t\geq 0. $$
\end{defo}

\begin{prop}[{\cite[Proposition II.1.7]{bensharp}}]
    If $f\in \M(R)$, its decreasing rearrangement $f^*$ is non-negative, non-increasing, right-continuous and has the following properties:
    \begin{enumerate}[label = (\alph*)]
        \item $|f| \leq |g|$ a.e. implies $f^*\leq g^*$,
        \item $(af)^* = |a| f^*$,
        \item for each $f,g\in \M(R)$ and $t_1, t_2 \geq 0$, $(f+g)^*(t_1 + t_2) \leq f^*(t_1) + g^*(t_2)$,
        \item $|f| \leq \liminf_n |f_n|$ implies $f^* \leq \liminf_n f_n^*$,
        \item $(|f|^p)^* = (f^*)^p$ for all $0 < p < \infty$.
    \end{enumerate}
\end{prop}

\begin{defo}
    A BFS $X(R)$ is called rearrangement-invariant (or r.i.) if for any $f\in X$ and $g\in \M(R)$ with $f^* = g^*$, it follows that $g\in X$ and $\|f\|_X = \|g\|_X$.
\end{defo}

R.i.\ spaces have a universal representation, given by Luxemburg's theorem:

\begin{thm}[{\cite[Theorem II.4.10]{bensharp}}]
    Let $X(R)$ be an r.i.\ space. Then, there exists an r.i.\ space $\overline{X}(\R^+)$ such that
    $$\|f\|_X = \|f^*\|_{\overline{X}}.$$
    When $R$ is non-atomic and has infinite measure, the representation $\overline{X}$ is unique. Moreover, the representation also satisfies $(\overline{X})' = \overline{X'}$, with equality of norms.
\end{thm}

This representation allows us to study the properties of r.i.\ spaces in terms of functions over $(0,\infty)$, which is sometimes more convenient. It also serves as a starting point to introduce Boyd indices, which are discussed in \cite[Section III.5]{bensharp}.

\begin{defo} \label{boyd}
    Let $X(R)$ be an r.i.\ space over a non-atomic space $R$ of infinite measure, and for $t>0$ let $E_t$ be the dilation operator on $\M(\R^+)$ acting as $E_t f(s) = f(st). $ Denote the function $h_X(t) := \|E_{\frac{1}{t}}\|_{\overline{X}\to\overline{X}}$, where $\overline{X}$ is the Luxemburg representation of $X$. We define the lower and upper Boyd indices of $X$, respectively, as the numbers
    $$\underline{\alpha}_X = \sup_{0 < t < 1} \frac{\log h_X(t)}{\log t}, \qquad \overline{\alpha}_X = \inf_{1 < t < \infty} \frac{\log h_X(t)}{\log t}. $$
\end{defo}

\begin{prop}[{\cite[Proposition III.5.13]{bensharp}}]
    If $X$ is an r.i.\ space, its Boyd indices can also be defined as
    $$ \underline{\alpha}_X = \lim_{t\to 0^+} \frac{\log h_X(t)}{\log t}, \qquad \overline{\alpha}_X = \lim_{t \to \infty} \frac{\log h_X(t)}{\log t}. $$
    They also have the following properties:
    $$ 0 \leq \underline{\alpha}_X \leq \overline{\alpha}_X \leq 1, \qquad \text{and} \qquad \underline{\alpha}_{X'} = 1 - \overline{\alpha}_X, \quad \overline{\alpha}_{X'} = 1 - \underline{\alpha}_X. $$
\end{prop}

When defining Boyd indices on completely atomic spaces, the non-uniqueness of the Luxemburg representation presents additional problems that require some workarounds in order to obtain a consistent definition. We give a simplified introduction in the case of r.i.\ spaces over $\Z$.

\begin{defo} \label{boyddiscr}
    Let $X(\Z)$ be an r.i.\ space, and denote $D$ the set of non-negative, non-increasing sequences $a = \{a_n\}_{n=1}^\infty$. Define the function
    $$ h_X(m) = \sup \{ \|E^d_m a\|_X \, : \, a\in D, \; \|a\|_X \leq 1 \}, \qquad m\geq 1,$$
    where $E^d_m$ is the discrete dilation operator acting as $(E^d_m a)_n = a_{mn}$. The lower Boyd index of $X$ is defined as
    $$ \underline{\alpha}_X = \sup_{m\geq 1} - \frac{\log h_X(m)}{\log m} = \lim_{m\to \infty} - \frac{\log h_X(m)}{\log m}, $$
    and the upper Boyd index can be defined as $\overline{\alpha}_X = 1 - \underline{\alpha}_{X'}. $
\end{defo}

\begin{rem}
    Note that in the continuous case we define $h_X(t)$ as the norm of $E_{\frac{1}{t}}$, while on the discrete case it is analogous to the norm of $E^d_t$.
\end{rem}

\subsection{Some useful operators}

Now, we introduce some operators that will be necessary in the proofs of the main theorems.

\begin{defo}
    For any $\lambda\in\R$, denote by $T_\lambda$ the translation operator that acts as
    $$T_\lambda f(x) = f(x-\lambda).$$
\end{defo}

Note that translations are isometries in r.i.\ spaces.

\begin{defo} 
    For $s>0$, define the moving average operators:
    \begin{equation} \label{movingavg}
        \begin{split}
        A_sf(x) = \frac{1}{s} &\int_x^{x+s} f(t) \, dt, \qquad A_{-s}f(x) = \frac{1}{s} \int^x_{x-s} f(t) \, dt, \\
        B_s f(x) &= \frac{1}{2s} \int^{x+s}_{x-s} f(t) \, dt = \frac{A_sf(x) + A_{-s}f(x)}{2}.
        \end{split}
    \end{equation}
\end{defo}

\begin{rem}
    All three of the operators above are convolution operators, and their kernels are, respectively:
    \begin{align*}
        a_s(t) &= \frac{1}{s} \chi_{(-s,0)}(t), \qquad a_{-s}(t) = \frac{1}{s} \chi_{(0,s)}(t), \qquad b_s(t) = \frac{1}{2s} \chi_{(-s,s)}(t).
    \end{align*}
\end{rem}

\begin{prop}
    Let $X$ and $Y$ be BFS over $\R$. If $B_s:X \to Y$ is bounded for every $s>0$, then for any bounded measurable set $E\subseteq \R$, there exists a constant $C_E > 0$ such that
    $$\|\chi_E * f\|_Y \leq C_E \|f\|_X;$$
    that is, the operator given by $ f \mapsto \chi_E * f $ is bounded from $X$ to $Y$.
\end{prop}

\begin{proof}
    Let $f\in X$ and suppose that $E \subseteq I = (-R,R)$. Taking the absolute value, we find that
    $$|\chi_E * f| \leq \chi_E*|f| \leq \chi_I*|f| = 2R B_R|f|. $$
    Taking norms and using monotonicity, $\|\chi_E*f\|_Y \leq 2R \|B_R\| \|f\|_X.$
\end{proof}

\begin{cor} \label{bbs}
    Let $X$ and $Y$ be BFS and suppose that $B_s$ is bounded for each $s>0$. Then, if $f\in L^\infty(\R)$ with bounded support, the convolution operator $g\mapsto g*f$ is bounded from $X$ to $Y$.
\end{cor}

\begin{cor}
    Let $X$ and $Y$ be BFS and let $s>0$. Then, $B_s:X\to Y$ is bounded if and only if both $A_s$ and $A_{-s}$ are bounded from $X$ to $Y$.
\end{cor}

We will also need a result for another kind of averaging operator.

\begin{thm}[{\cite[Proposition II.4.8]{bensharp}}]\label{avgri}
    Let $R$ be a measure space that is either non-atomic or completely atomic with all atoms of equal measure, and let $\{E_j\}_{j\in\N}$ be a countable family of disjoint subsets of $R$, with $0 < \mu(E_j) < \infty$ for every $j\in\N$. Let $E = R\setminus \bigcup_j E_j$, and define the operator
    $$ Vf(x) = \chi_E(x) f(x) + \sum_{j\in\N} \frac{1}{\mu(E_j)} \int_{E_j} f(y) \, d\mu(y) \cdot \chi_{E_j}(x). $$

    Then, for any r.i.\ space $X$ over $R$, the operator $V$ is bounded on $X$ with norm $\|V\|_{X\to X}\leq 1$.
\end{thm}

\section{An initial characterization on Banach function spaces} \label{sec:bfs}

In this section, our goal is to prove that, under certain general conditions, the boundedness of the segment multiplier is equivalent to that of the truncated Hilbert transform on Banach function spaces (see Theorem \ref{charbfs}). To this end, we will consider BFS over $\R$. We recall the definitions of the operators involved.

\begin{defo}
    We define the following kernel functions and their associated convolution operators: 
    {\small
    $$\begin{array}{llll}
        s(t) = \frac{\sin t}{t}  &\Longrightarrow&  Sf = s*f & \text{(Segment multiplier)} \medskip \\
        h(t) = \frac{1}{t}  &\Longrightarrow & Hf(x) = \mathrm{p.v.}\int_\R \frac{f(y)}{x-y} \, dy & \text{(Hilbert transform)} \medskip \\
        h_\eps(t) = \frac{1}{t}\chi_{\R\setminus(-\eps,\eps)}(t)  &\Longrightarrow & H_\eps f = h_\eps *f & \text{(Truncated Hilbert transform)}
    \end{array}$$} \vspace{-2pt}
\end{defo}

\begin{rem} \label{rem:trunceq}
    Let $X$ and $Y$ be BFS over $\R$. Suppose that $B_s : X\to Y$ is bounded for any $s>0$. Then, given $\eps,\delta>0$, Corollary \ref{bbs} implies that $H_\eps - H_\delta : X\to Y$ is bounded. Thus, the boundedness of $H_\eps$ and $H_\delta$ are equivalent.
\end{rem}

\begin{defo} 
    Let $P$ be the convolution operator with respect to the Poisson kernel, $\frac{1}{x^2 + 1}$.
\end{defo}

\begin{lemma}\label{poisson}
    Let $X, Y \subseteq L^1_{\loc}(\R)$ be Banach spaces, and assume that $\|\cdot\|_Y$ is monotone and has the Fatou property. Suppose that the operators $\{B_s\}_{s>1}$ are uniformly bounded from $X$ to $Y$, $\|B_s\|\leq C$, with $C$ independent of $s>1$. Then, $P:X\to Y$ is bounded.
\end{lemma}

\begin{proof}
    For $j\in \N$, let $b_j = \frac{\chi_{(-j,j)}}{2j}$ be the convolution kernel of $B_j$, and define $g_n = \sum_{j=1}^n b_j/j^2. $ Since $\{B_j\}_j$ are uniformly bounded, the operator $G_n f = g_n * f$ is bounded. Moreover, $\{g_n*f\}_n$ is a Cauchy sequence in $Y$ and we can define the operator $Gf := \lim_n g_n*f$. Because of the Fatou property, we get that
    $$\|Gf\|_Y \leq \liminf_n \|g_n * f\|_Y \leq C \sum_{j=1}^\infty \frac{1}{j^2} \|f\|_X, $$
    so $G$ is bounded. Using the monotone convergence theorem, we can see that $G$ is the convolution operator with kernel $g = \sum_{j=1}^\infty \chi_{(-j,j)}/2j^3. $

    Note that for $k\geq 1$ we can compare
    $$ \sum_{j=k}^\infty \frac{1}{j^3} \approx \int_k^\infty \frac{1}{x^3} \, dx = \frac{1}{2 k^2} \approx \frac{1}{k^2 + 1}.$$ 
    Then, if $|x| \in (k,k+1)$, we can see that
    $$g(x) = \sum_{j=k+1}^\infty \frac{1}{2j^3} \approx \frac{1}{k^2 + 1} \approx \frac{1}{x^2 + 1}.$$

    Because both $G$ and $P$ have positive kernels, we have that
    $$ |Gf| \leq G(|f|), \quad |Pf| \leq P(|f|), $$
    and also, for every non-negative function $f$, $ Pf \approx Gf. $ Hence, we obtain that
    $$ \|Pf\|_Y \leq \|P(|f|)\|_Y \approx \|G(|f|)\|_Y \lesssim \|f\|_X, $$
    so $P$ is bounded from $X$ to $Y$.
\end{proof}

We need a lemma regarding the interplay of translations and the truncated Hilbert transform.

\begin{lemma} \label{tras}
    Let $X$ and $Y$ be BFS such that the translations $T_k$ are bounded on $Y$ for any $k\in\R$, and suppose that the operators $\{B_s\}_{s>1}$ are uniformly bounded from $X$ to $Y$. Then, given any $r > 0$, $k\in \R$, $E\subseteq \R$ and any bounded sequence $\{f_n\}_n\subseteq X$, we have that the boundedness in $Y$ of the following sequences is equivalent:
    $$\{H_r f_n \cdot \chi_E\}_n \; \Longleftrightarrow \; \{H_r f_n \cdot \chi_{E+k}\}_n. $$
\end{lemma}

\begin{proof}
    For any $r,s>0$, it is straightforward to check that $\{\|H_{r} f_n \cdot \chi_E\|_Y\}_n$ is bounded if and only if $\{\|H_{s} f_n \cdot \chi_E\|_Y\}_n$ is.

    For any fixed $k>0$, consider now the kernels
    $$h_{2k}(x) = \frac{1}{x}\chi_{\R\setminus (-2k,2k)}(x), \qquad F_{k}(x) = \frac{1}{x + k}\chi_{\R\setminus (-2k,2k)}(x). $$
    A quick computation gives that for any $|x|> 2k$ we have
    \begin{equation} \label{eq:eqtrans}
        \frac{1}{x} - \frac{1}{x+k} \approx \frac{1}{x^2 + 1},
    \end{equation}
    where the equivalence constants depend on $k$.
    Since $P:X \to Y$ is a bounded operator, using Lemma~\ref{poisson}, it follows that $\{H_{2k} f_n \cdot \chi_E\}_n$ is bounded in $Y$ if and only if $\{(F_k * f_n) \cdot \chi_E\}_n$ is.
    
    If we write the convolution with $F_k$ for a function $f$, we get
    \begin{align*}
        F_k * f(x) &= \int_{|x-y|>2k} \frac{f(y)}{x+k-y} \, dy \\
        &= \int_{|x+k-y|>3k} \frac{f(y)}{x+k-y} \, dy + \int_{x+2k \leq y \leq x+4k} \frac{f(y)}{x+k-y} \, dy \\
        &= H_{3k} f(x+k) + \int_{x+2k \leq y \leq x+4k} \frac{f(y)}{x+k-y} \, dy \\ 
        &= T_{-k} H_{3k} f(x) + \int_{x+2k \leq y \leq x+4k} \frac{f(y)}{x+k-y} \, dy.
    \end{align*}
    Note that the second term can be bounded as
    $$ \left| \int_{x+2k \leq y \leq x+4k} \frac{f(y)}{x+k-y} \, dy \right| \leq \int_{x+2k}^{x+4k} \frac{|f(y)|}{y - x -k} \, dy \leq 2 B_{k}(|f|)(x+3k) = 2 T_{-3k}B_{k}(|f|)(x), $$
    and the operator $T_{-3k}B_{k}$ is bounded. 
    
    For the other term, multiplying by the characteristic function of a set $E$, we obtain the following equality:
    \begin{align*}
        T_{-k} H_{3k} f(x) \cdot \chi_E(x)  &= H_{3k} f(x+k) \cdot \chi_{E+k}(x+k) = T_{-k}(H_{3k} f \cdot \chi_{E+k})(x).
    \end{align*}
    Taking norms, since translations are isomorphisms, for any set $E$ and any $k> 0$, we get
    $$\|T_{-k} H_{3k} f\cdot \chi_E\|_Y \approx \|H_{3k} f \cdot \chi_{E+k}\|_Y,$$
    where the constants in the equivalence may depend on $k$. The same reasoning can be repeated for the kernels $F_{-k}(x) = \frac{1}{x - k}\chi_{\R\setminus (-2k,2k)}(x)$.
    
    Thus, assuming that $k \neq 0$, we have that the following sequences have the same boundedness or unboundedness with respect to $\|\cdot\|_Y$:
    \begin{align*}
        \{H_{2|k|} f_n \cdot \chi_E\}_n \; &\Longleftrightarrow \; \{(F_k * f_n) \cdot \chi_E\}_n \; \Longleftrightarrow \; \{T_{-k} H_{3|k|} f_n \cdot \chi_E\}_n \\
        &\Longleftrightarrow \; \{H_{3|k|}f_n \cdot \chi_{E+k}\}_n.
    \end{align*}
    Recalling how we started the proof, we can deduce that $\{\|H_{r} f_n \cdot \chi_E\|_Y\}_n$ is bounded if and only if $\{\|H_{r} f_n \cdot \chi_{E+k}\|_Y\}_n$ is, for any $k\in\R$.
\end{proof}

We will now prove the characterization on BFS.

\begin{thm}\label{charbfs}
    Let $X$ and $Y$ be BFS where the translations $T_\lambda:Y \to Y$ are bounded for any $\lambda\in\R$, that the operators $B_s: X \to Y$ are uniformly bounded for $s >1$ and that there exists $\delta > 0$ such that $B_\delta: X \to X$ is bounded. 
    
    Then, the segment multiplier $S:X\to Y$ is bounded if and only if the truncated Hilbert transform $H_1:X\to Y$ is bounded.
\end{thm}

We will begin by showing that $S$ is bounded if $H_1$ is, since the converse is more involved.

\begin{proof}
    Suppose $H_1$ is bounded. We can expand the definition of the operator $S$ as follows:
    \begin{align*}
        Sf(x) &= \int_\R \frac{\sin(x-y)}{x-y}f(y) \, dy = \int_\R \frac{\sin x\cos y - \cos x  \sin y }{x-y}f(y) \, dy \\
        &= \int_{|x-y|<1} \frac{\sin(x-y)}{x-y}f(y) \, dy + \sin x \int_{|x-y|>1} \frac{\cos y }{x-y}f(y) \, dy \; \\
        &\quad - \cos x \int_{|x-y|>1} \frac{ \sin y }{x-y}f(y) \, dy \\ 
        &= \int_{|x-y|<1} \frac{\sin(x-y)}{x-y}f(y) \, dy\; + \sin x \,  H_1(\cos\cdot f)(x) - \cos x \, H_1(\sin \cdot f)(x).
    \end{align*}

    Let us bound each of the three terms in the right-hand side. Note that $\left|\frac{\sin x}{x}\right| \leq 1$. Since the first term is the convolution with $\frac{\sin x}{x} \chi_{(-1,1)}(x)$, it is bounded by Corollary \ref{bbs}.

    As for the second and third terms, both follow the same reasoning, which we show with the second one:
    $$\|\sin\cdot H_1(\cos \cdot f)\|_Y \leq \| H_1(\cos \cdot f)\|_Y \lesssim \|\cos \cdot f\|_X \leq \|f\|_X.$$
    Since we have bounded all three terms, it follows that $\|Sf\|_Y \lesssim \|f\|_X.$

    Let us prove the converse by contradiction. Suppose that $S$ is bounded but the operator $H_1$ is unbounded, and hence $H_r$ is unbounded for every $r>0$ (as seen in Remark \ref{rem:trunceq}). Because of this, we can find a sequence $\{f_n\}_{n\in\N} \subseteq X$ with $\|f_n\|_X\leq 1$, for all $n$, such that
    \begin{equation} \label{unbdd} 
        \|H_1 f_n\|_Y \underset{n}{\longrightarrow} \infty.
    \end{equation}
    Now, if we put $f_n = f_n^+ - f_n^-$, with $f_n^+$ and $f_n^-$ being the positive and negative parts of $f_n$, respectively, either $\{\|H_1f_n^+\|_Y\}_n$ or $\{\|H_1f_n^-\|_Y\}_n$ is unbounded. Otherwise, the triangle inequality would imply that $\{H_1 f_n\}_n$ is bounded in $Y$. Then, by taking an appropriate subsequence of the unbounded one, we can assume that the sequence satisfying (\ref{unbdd}) consists of non-negative functions.

    Now, let $r\in \N$ with $r \geq 4$ such that $\eps := 2\pi / r < \delta$. For $1\leq l \leq r$, define the sets
    \begin{equation}\label{j^l}
        I^l = [(l-1)\eps, l\eps), \qquad \text{and} \qquad J^l = \bigcup_{j\in\Z} (2\pi j + I^l).
    \end{equation}
    Now, since the family $\{J^l\}_{l=1}^r$ is a finite partition of $\R$, we can put $f = \sum_{l=1}^r f\chi_{J^l}. $ With this, the triangle inequality implies again that there must exist an $l_0$ such that $\{ \| H_1( f_n \chi_{J^{l_0}} ) \|_Y \}_n$ is unbounded. Therefore, by taking another subsequence, we can assume that the functions $\{f_n\}_n$ satisfying (\ref{unbdd}) are non-negative and vanish outside of $J^{l_0}$. From now on, this index $l_0$ will be fixed, so denote
    \begin{equation} \label{j_1}
         I_1 = I^{l_0}, \qquad J_1 = J^{l_0} = \bigcup_{j\in\Z} (2\pi j + I_1). 
    \end{equation}

    Consider the floor function, $\lfloor \cdot \rfloor$, which satisfies $\lfloor x \rfloor \leq x < \lfloor x \rfloor +1$ for every $x\in\R$. Then, recalling (\ref{eq:eqtrans}), we can see that for $|x|\geq 2$, 
    $$\left| \frac{1}{\lfloor x\rfloor} - \frac{1}{x} \right| \lesssim \frac{1}{x^2 + 1}.$$
    Define the convolution operator $K$ with kernel
    \begin{equation} \label{defk}
        k(x) = \frac{1}{\pi \lfloor \frac{x}{\pi} \rfloor} \chi_{\R\setminus(-2\pi, 2\pi)}(x),
    \end{equation}
    so the previous argument results in
    \begin{equation} \label{eq:difker}
        |k(x) - h_{2\pi}(x)| \lesssim \frac{1}{x^2 + 1}\chi_{\R\setminus(-2\pi, 2\pi)}(x) \leq \frac{1}{x^2 + 1}.
    \end{equation}
    Thus, since the operators $\{B_s\}_{s>1}$ are uniformly bounded, Lemma~\ref{poisson} gives that the convolution with the Poisson kernel is bounded, and so $H_{2\pi}$ is bounded from $X$ to $Y$ if and only if $K$ is.

    Let us see what the image of $f_n$ through $K$ looks like: recalling the definition of $J_1$ in (\ref{j_1}), we have
    \begin{equation} \label{imagekpi}
    \begin{split}
        Kf_n(x) = \int_{J_1} \frac{f_n(y)}{\pi \lfloor \frac{x-y}{\pi}\rfloor} \chi_{\R\setminus(-2\pi, 2\pi)}(x-y) \, dy = \sum_{j\in\Z} \int_{(2\pi j + I_1)\cap (|x-y|>2\pi)} \frac{f_n(y)}{\pi \lfloor \frac{x-y}{\pi}\rfloor} \, dy.
    \end{split}
    \end{equation}
    Note that the function $y \mapsto \lfloor \frac{y}{\pi}\rfloor$ is constant over the intervals $[l\pi, (l+1)\pi)$, so the function $y \mapsto \lfloor \frac{x-y}{\pi}\rfloor$ is constant over each $(x-(l+1)\pi, x-l\pi]$. Then, if we require that for each $j$ there exists an $l$ so that $(2\pi j + I_1) \subseteq (x-(l+1)\pi, x-l\pi] $, the function $Kf_n$ would be constant on the domain of each integral in the series in (\ref{imagekpi}). If we write the interval $I_1 = (a,a+\eps)$, a quick computation shows that the inclusion holds if and only if there exists an $m \in \Z$ such that $x\in \big(m\pi + (a +\eps, a + \pi)\big).$ For such $x$, it also holds that the function $y \mapsto \chi_{(-2\pi,2\pi)}(x-y)$ is constant on each $2\pi j + I_1$.

    With this, defining the functions
    \begin{equation*}
        g_n(x) = \sum_{j\in\Z} \bigg(\frac{1}{|I_1|} \int_{2\pi j +I_1} f_n(t) \, dt\bigg) \chi_{2\pi j + I_1}(x), 
    \end{equation*}
    we have for every $x\in \bigcup_{m\in\Z} \big(m\pi + (a+\eps,a+\pi)\big) $ the equality $K g_n (x) = K f_n (x).$

    Let us check that $g_n \in X$. Consider the image of $f_n$ under the averaging operator $B_\eps$: if $x\in 2\pi j + I_1$, we have that
    $$2\eps B_\eps f_n(x) = \int_{x-\eps}^{x+\eps} f_n(t)\, dt = \int_{2\pi j+I_1} f_n(t)\, dt = \eps g_n(x). $$
    Hence, $g_n = \chi_{J_1} \cdot 2 B_\eps f_n, $ so noting that $2 \eps B_\eps (|f_n|) \leq 2\delta B_\delta (|f_n|) $, we have that $g_n\in X$ with $\|g_n\|_X \leq \|2 B_\eps f_n\|_X \leq C_\eps \|f_n\|_X.$

    Now, let us start the comparison with the segment multiplier. Since we already know (\ref{eq:difker}),
    we can multiply the left-hand side by $\sin x $ to get, for $|x| \geq 2\pi$,
    $$\left| \frac{\sin x }{x} - \frac{\sin x}{\pi \lfloor \frac{x}{\pi}\rfloor} \right| \lesssim \frac{1}{x^2 + 1}. $$
    Hence, from the boundedness of the segment multiplier and $P$ it follows that the convolution with $s_\pi (x) = \frac{\sin x}{\pi \lfloor \frac{x}{\pi}\rfloor} \chi_{(-2\pi, 2 \pi)^c}(x)$ is bounded.

    Considering the reasoning followed after (\ref{imagekpi}), we know that, if we take $y\in (2\pi j + I_1)$ and $x \in \big(2\pi l + (a+\eps,a+\pi)\big)$, then $\left\lfloor \frac{x-y}{\pi} \right\rfloor = 2(l-j).$ Put $g_{n,j} = \frac{1}{|I_1|} \int_{2\pi j + I_1} f_n(t) \, dt $ for each $n\in\N$ and $j\in\Z$, so that $g_n = \sum_{j\in\Z} g_{n,j} \cdot \chi_{2\pi j + I_1} $. Thus, we can obtain through some computations that
    \begin{align*}
        Kg_n(x) = \eps \sum_{j\not\in\{l-1,l\}} \frac{g_{n,j}}{2\pi (l-j)} .
    \end{align*}

    Arguing the same way with $S_\pi$, we get
    \begin{align*}
        S_\pi g_n(x) = \left(\int_{x-a-\eps}^{x - a} \sin(y) \, dy \right) \sum_{j\not\in\{l-1,l\}} \frac{g_{n,j}}{2\pi (l-j)}.
    \end{align*}

    The integral on the last term vanishes whenever $x = \pi m + a + \frac{\eps}{2}$. Since we are taking $x\in \big(2\pi l + (a + \eps, a + \pi)\big)$, this can never happen, so in this interval
    \begin{equation} \label{eqsin}
        \left |\int_{x-a-\eps}^{x - a} \sin y \, dy \right | \approx \eps.
    \end{equation}
    Therefore, for any $x\in \big(2\pi l + (a + \eps, a + \pi)\big)$, $|K g_n(x)| \approx |S_\pi g_n(x)| $ with the same constants as in the equivalence in (\ref{eqsin}).

    Let $J_2 = \bigcup_{l\in\Z} 2\pi l + (a+\eps, a+\pi) $. Recalling that $\eps = 2\pi/r$, the condition $r\geq 4$ allows $(a + \eps, a + \pi)$ to have larger measure than $(a, a+\eps)$. Since the constants on the equivalence (\ref{eqsin}) are independent of $n$, we have that $\{Kg_n \cdot \chi_{J_2}\}_n$ is bounded if and only if $\{S_\pi g_n \cdot \chi_{J_2}\}_n$ is.

    On one hand, recalling that $\|g_n\|_X \lesssim \|f_n\|_X$, we have that
    $$\|S_\pi g_n \cdot \chi_{J_2}\|_Y \leq \|S_\pi g_n\|_Y \lesssim \|g_n\|_X \lesssim \|f_n\|_X \leq 1 $$
    independently of $n$, so $\{S_\pi g_n \cdot \chi_{J_2}\}_n$ is a bounded sequence in $Y$.
    
    On the other hand, tracing back our steps, we know that $\{Kg_n \cdot \chi_{J_2}\}_n$ is bounded if and only if $\{H_1 f_n \cdot \chi_{J_2}\}_n$ is. To obtain a contradiction, let us check that this last sequence is unbounded.

    Recall the definition of the sets $J^l$ in (\ref{j^l}), with the family $\{J^l\}_{l=1}^r$ being a finite partition of $\R$. Arguing the same way as in the beginning, we can write $ H_1 f_n = \sum_{l=1}^r H_1 f_n \cdot \chi_{J^l}. $ Then, there must exist a particular $l_1$ such that $\{\|H_1 f_n \cdot \chi_{J^{l_1}}\|_Y\}_n$ is unbounded, or else $\{\|H_1 f_n\|_Y\}_n$ would be bounded. In particular, there exists a subsequence of $\{\|H_1 f_n \cdot \chi_{J^{l_1}}\|_Y\}_n$ diverging to infinity. Because of Lemma~\ref{tras}, we know that translating the characteristic function in this sequence preserves its divergence. Having this in mind, take $J_3 = \bigcup_{j\in\Z} \big(2\pi j + (a+\eps, a + 2\eps)\big), $ so $\{\|H_1 f_n \cdot \chi_{J_3}\|_Y\}_n$ is unbounded. Since $r\geq 4$, $2\eps \leq \pi$ and we have $J_3 \subseteq J_2$. By the monotonicity of the norm, the divergence of $\{\|H_1 f_n \cdot \chi_{J_3}\|_Y\}_n$ implies the same behavior for $\{\|H_1 f_n \cdot \chi_{J_2}\|_Y\}_n$.

    However, from this last fact it follows that $\{\|K g_n \cdot \chi_{J_3}\|_Y\}_n$ is unbounded, which is a contradiction. Thus, $H_1$ is a bounded operator from $X$ to $Y$.
\end{proof}

\begin{rem}
    Note that when proving that $H_1$ bounded implies $S$ bounded, we did not use the hypotheses that $T_\lambda:Y \to Y$ or $B_\delta: X \to X$ are bounded operators.
\end{rem}

In the diagonal case $X = Y$, we can state the following corollary:

\begin{cor}
    Let $X$ be a BFS where the translations $T_\lambda$ are bounded for any $\lambda\in\R$, and suppose that the operators $\{B_s\}_{s>1}$ are uniformly bounded on $X$.
    
    Then, $S$ is bounded on $X$ if and only if $H_1$ is.
\end{cor}

\section{The main theorem} \label{sec:ri}

In this section, we prove the main result of the paper, Theorem \ref{mainri}, that gives the characterization of the boundedness of the segment multiplier. To do this, we will show that r.i.\ spaces satisfy the hypotheses on Theorem \ref{charbfs}, and we will use the richer properties of r.i.\ spaces to introduce a new sequence space $\ell_X$, which allows us to establish the connection between the truncated and the discrete Hilbert transform.

Let us begin with the following definitions that we will use throughout this section:

\begin{defo}
    Let $X$ be an r.i.\ space over $\R$. We define its discretized space as:
    $$\ell_X = \ell_X(\Z) := \bigg\{ (a_n)_{n\in\Z} \, :\, Ba = \sum_{n\in\Z} a_n \chi_{(n-1,n)} \in X \bigg\}.$$

    For a sequence $(a_n)_n \in \ell_X$, we define its norm as $\|(a_n)_n\|_{\ell_X} = \big\| \sum_{n\in\Z} a_n \chi_{(n-1,n)} \big \|_X. $
\end{defo}

\begin{prop} \label{prop:ellx}
    The space $\ell_X$ is an r.i.\ space.
\end{prop}

\begin{proof}
    All the properties of being a Banach function space are inherited from the norm of $X$. To check the r.i.\ property, it suffices to observe that the decreasing rearrangement of $Ba$ agrees with that of $(a_n)_n$ (viewed as a step function on $\R^+$). With this, the norm $\|(a_n)_n\|_{\ell_X} = \|Ba\|_X$ only depends on $(Ba)^*$, and $\ell_X$ is r.i.
\end{proof}

There are two natural operators that relate $X$ to $\ell_X$.

\begin{defo} 
    Define the pair of operators $X \overset{A}{\longrightarrow} \ell_X \overset{B}{\longrightarrow} X,$ that act on functions and sequences as
    \begin{equation} \label{ba}
        Af = \left\{ \int_{j-1}^j f(y)\, dy \right\}_{j\in\Z} \quad \text{and} \qquad Ba = \sum_{j\in\Z} a_j \chi_{(j-1,j)}.
    \end{equation}
\end{defo}

\begin{prop}
    The operators $A$ and $B$ are bounded, with $AB = I_{\ell_X}$ and $\|BA\|_{X\to X} \leq 1$.
\end{prop}

\begin{proof}
    The operator $B$ is trivially an isometric embedding. $A$ is well-defined and bounded if and only if $BA$ is, by definition of $\ell_X$. The fact that it is bounded is given by Theorem~\ref{avgri}.
\end{proof}

\begin{rem}
    Using the operator $A$, one can prove that $(\ell_X)' = \ell_{X'}$.
\end{rem}

Many other elementary properties of the space $\ell_X$ can be obtained easily, so we will only focus on the ones relevant to our results.

\begin{prop} \label{prop:sumint}
    Let $X$ and $Y$ be r.i.\ spaces. Then,
    $$ \ell_{X\cap Y} = \ell_X \cap \ell_Y, \quad \text{and} \quad \ell_{X + Y} = \ell_X + \ell_Y. $$

    In particular, $\ell_{X + L^1} = \ell_{X\cap L^\infty} = \ell_X$.
\end{prop}

\begin{proof}
    Observe that the inclusion $\ell_X + \ell_Y \subseteq \ell_{X+Y}$ is trivial: if $a\in \ell_X + \ell_Y$, there exist $a_1 \in \ell_X$ and $a_2 \in \ell_Y$ such that $a = a_1 + a_2$, so
    $$ a = ABa = A(Ba_1 + Ba_2) \in A(X + Y) = \ell_{X+Y}. $$
    For the other inclusion, similarly, if $a \in \ell_{X+Y}$, then $Ba \in X + Y$. Let $f\in X$ and $g\in Y$ such that $Ba = f + g$. Applying the operator $A$, we have that
    $$ a = ABa = Af + Ag \in \ell_X + \ell_Y, $$
    and the equality is proved. Observe that set equality of BFS implies the equivalence of their norms.

    For the intersection, recalling that $\|f\|_{X\cap Y} = \max\{\|f\|_X, \|f\|_Y\}$, and since for any sequence $a$ its norm is defined as $\|a\|_{\ell_X} = \|Ba\|_X$, we have that
    $$ \|a\|_{\ell_{X \cap Y}} = \|Ba\|_{X \cap Y} = \max\{ \|Ba\|_X, \|Ba\|_Y \} = \max\{ \|a\|_{\ell_X}, \|a\|_{\ell_Y} \} = \|a\|_{\ell_X \cap \ell_Y}. $$

    The last statement follows from the fact that every r.i.\ space $W$ over $\Z$ satisfies $\ell^1 \subseteq W \subseteq \ell^\infty$ (see \cite[Theorem II.6.6]{bensharp}).
\end{proof}

In fact, the space $\ell_X$ also represents the norm of $X \cap L^\infty$ in the following way:

\begin{prop} 
    Let $X$ be an r.i.\ space. Then, $f\in X\cap L^\infty$ if and only if $\{f^*(k)\}_{k=0}^\infty \in \ell_X$ and, in this case,
    $$ \|f\|_{X\cap L^\infty} \approx \|\{f^*(k)\}_{k=0}^\infty\|_{\ell_X}. $$
\end{prop}

\begin{proof}
    Note that we have
    \begin{align*}
        \sum_{k=0}^\infty f^*(k) \chi_{(k-1,k)} = \|f\|_\infty \chi_{(-1,0)} + B(\{f^*(k)\}_{k=1}^\infty).
    \end{align*}
    Now, it is clear that $f^*(k) \leq \int_{k-1}^k f^*(t)\, dt$, so $B(\{f^*(k)\}_{k=1}^\infty) \leq BA(f^*)$. Hence, taking $f\in X\cap L^\infty$,
    \begin{align*}
        \| \{f^*(k)\}_{k=0}^\infty \|_{\ell_X} &= \Big\|\sum_{k=0}^\infty f^*(k) \chi_{(k-1,k)}\Big\|_X \lesssim \|f\|_\infty + \|BA(f^*)\|_{\overline{X}} \\
        &\lesssim \|f\|_\infty + \|f\|_X \approx \|f\|_{X\cap L^\infty}.
    \end{align*}

    The other inequality follows from
    $$ f^*(x) \leq \sum_{k=0}^\infty f^*(k) \chi_{[k,k+1)}. $$
\end{proof}

\begin{cor} \label{cor:ellxlinfty}
    Let $X$ and $Y$ be r.i.\ spaces. Then, $\ell_X \subseteq \ell_Y$ if and only if $X\cap L^\infty \subseteq Y \cap L^\infty$.
\end{cor}

With the help of the discretized space, we define a new class of indices for any r.i.\ space derived from Definition~\ref{boyddiscr}.

\begin{defo}
    Let $X(\R)$ be an r.i.\ space and $\ell_X$ its discretization. The discrete Boyd indices of $X$ are defined as the Boyd indices of $\ell_X$,
    $$\underline{d}_X := \underline{\alpha}_{\ell_X}, \qquad \overline{d}_X := \overline{\alpha}_{\ell_X}.$$
\end{defo}

\begin{prop} 
    The discrete indices of an r.i.\ space $X$ satisfy
    $$0 \leq \underline{\alpha}_X \leq \underline{d}_X \leq \overline{d}_X \leq \overline{\alpha}_X \leq 1. $$
\end{prop}

\begin{proof}
    Recall the definitions of the continuous and discrete dilation operators, denoted respectively as $ E_\lambda f(x) = f(\lambda x)$ and $ (E^d_m a)_n = a_{mn}, $ for any function $f$ and sequence $a = (a_n)_n$. If $a=\{a_n\}_{n\in\N}$ is a non-negative, non-increasing sequence, we want to compare the real function $(E^d_m a)^*$ to $E_m (a^*)$, for $m \in\N$.

    We have that $E^d_m a = \{a_{mn}\}_{n\in\N}$, so its rearrangement is $(E^d_m a)^* = \sum_{n=1}^\infty a_{mn} \chi_{(n-1,n)}. $
    On the other hand, we have that
    $$E_m \Big(\sum_{n=1}^\infty a_n \chi_{(n-1,n)}\Big) = \sum_{n=1}^\infty a_n E_m \chi_{(n-1,n)} = \sum_{n=1}^\infty a_n \chi_{(\frac{n-1}{m},\frac{n}{m})}. $$

    Suppose $x\in (\frac{n-1}{m},\frac{n}{m})$, for some $n\geq 1$. It is clear that $E_m a^* (x) = a_n$. In the case that $n \not \in m\N$, we have the inclusion $(\frac{n-1}{m},\frac{n}{m}) \subseteq (\lfloor \frac{n}{m}\rfloor, \lfloor \frac{n}{m}\rfloor +1) $, and then $(E^d_m a)^*(x) = a_{m(\lfloor\frac{n}{m}\rfloor + 1)}. $
    Now, notice that $m ( \left\lfloor n/m \right\rfloor +1 ) \geq n, $ so, because $a$ is non-increasing,
    $$(E^d_m a)^*(x) = a_{m\lfloor \frac{n}{m}\rfloor +m} \leq a_n = E_m a^* (x).$$
    In the case that $n = ml$ for some $l\in\N$, we have $x \in (l - \frac{1}{m}, l) \subseteq (l-1,l)$. Thus,
    $$(E^d_m a)^*(x) = a_{ml} = a_n = E_m a^* (x). $$

    Then, when we consider the norms for these operators and recalling the $h$ functions from Definitions~\ref{boyd} and \ref{boyddiscr}, we see that
    \begin{align*}
        h_{\ell_X}(m) &= \sup\bigg\{ \frac{\|E^d_m a\|_{\ell_X}}{\|a\|_{\ell_X}} \, : \, \|a\|_{\ell_X} \leq 1, \, a\text{ positive and non-increasing} \bigg\} \\
        &\leq \sup\bigg\{ \frac{\|E_m (a^*)\|_{X}}{\|a\|_{\ell_X}} \, : \, \|a\|_{\ell_X} \leq 1, \, a\text{ positive and non-increasing} \bigg\} \\
        & \leq \|E_m\| = h_X(m).
    \end{align*}

The result follows from this inequality and the definition of Boyd indices.
\end{proof}

Let us show how these indices might behave.

\begin{ex}\label{exdiscrind}
    For any $1 \leq p \leq \infty$, it is straightforward that $\ell_{L^p} = \ell^p$, so the discrete indices of $L^p$ are $ \underline{d}_{L^p} = \overline{d}_{L^p} = \frac{1}{p}, $ which coincide with its regular Boyd indices.

    Consider $1 \leq p < q \leq \infty$ and the space $L^p + L^q$ with norm
    $$\|f\|_{L^p + L^q} = \inf\{ \|g\|_p + \|h\|_q \, : \, f = g + h, \; g\in L^p, \; h\in L^q \}. $$
    This space is r.i.\ and it has indices $ \underline{\alpha}_{L^p + L^q} = 1/q $ and $\overline{\alpha}_{L^p + L^q} = 1/p. $ Note that $Af \in \ell^\infty$, for any $f\in L^r$. When discretizing the space, if $f = f_1 + f_2$ for some $f_1 \in L^p$ and $f_2 \in L^q$, then $Af_1 \in \ell^p \subseteq \ell^q$, so $Af \in \ell^q$. Therefore, since it is easy to see that $\ell^q \subseteq \ell_{L^p + L^q}$, we have the equality $\ell_{L^p + L^q} = \ell^q,$ and the discrete indices are both
    $$ \underline{d}_{L^p + L^q} = \overline{d}_{L^p + L^q} = \frac{1}{q} < \frac{1}{p} = \overline{\alpha}_{L^p + L^q}. $$
\end{ex}

Let us define the discrete Hilbert transform, as it plays a key role in the proof of our main result, Theorem \ref{mainri}.

\begin{defo}
    Let $a = (a_k)_{k\in\Z}$ be a sequence. Its discrete Hilbert transform is the sequence $\H a$ with general term
    $$ (\H a)_k = \sum_{j\neq k} \frac{a_j}{k-j}, $$
    if the series converges.
\end{defo}

With all the relevant definitions introduced, we state the main theorem of this section:

\begin{thm} \label{mainri}
    Let $X(\R)$ and $Y(\R)$ be r.i.\ spaces. Then, the following are equivalent:
    \begin{enumerate}[label=(\alph*)]
        \item \label{mainria} the segment multiplier $S: X \to Y$ is bounded,
        \item \label{mainrib} the truncated Hilbert transform $H_1: X \to Y$ is bounded,
        \item \label{mainric} the discrete Hilbert transform $\H: \ell_X \to \ell_Y$ is bounded.
    \end{enumerate}
    
\end{thm}

\begin{proof}
We will split the proof of the theorem into several lemmas. We start by proving that, under any of the conditions in the statement, the operators $B_s$ are bounded from $X$ to $Y$ uniformly on $s > 1$.

\begin{lemma} \label{convl1}
    Let $X$ be an r.i.\ space, and $g\in L^1$. Then, if $Gf = g*f$, the operator $G: X \to X$ is bounded with $\|G\|_{X\to X} \leq \|g\|_1$.

    In particular, $\|B_s\|_{X\to X} \leq 1$ for every $s > 0$.
\end{lemma}

\begin{proof}
    Note that $G$ is bounded on $L^p$ for every $1 \leq p \leq \infty$, with $\|G\|_{L^p \to L^p} = \|g\|_1 $. In particular, we have it for $p = 1,\infty$. Because of \cite[Theorem III.2.2]{bensharp}, we know that $G: X \to X$ is bounded and $\|G\|_{X\to X} \leq \max\{ \|G\|_{L^1\to L^1}, \|G\|_{L^\infty\to L^\infty} \} = \|g\|_1.$
\end{proof}

Note that, using H\"older's inequality, the r.i.\ property of $\|\cdot\|_X$, and \cite[Theorem II.5.2]{bensharp}, we obtain that 
$$ |B_s f(x)| = \bigg| \frac{1}{2s} \int_{x-s}^{x+s} f(y) \, dy \bigg| \leq \frac{1}{2s} \|\chi_{(x-s,x+s)}\|_{X'} \|f\|_X = \frac{1}{\|\chi_{(0,2s)}\|_{X}} \|f\|_X. $$

Thus, we obtain that $B_s: X \to X\cap L^\infty$ is bounded, and the constant is uniform over $s > 1$.

\begin{lemma} \label{lem:interplin}
    Let $X$ be an r.i.\ space, and $f$ be a function that satisfies
    \begin{equation} \label{interplin}
        f(x) = (1-(x-\lfloor x \rfloor)) f(\lfloor x \rfloor) + (x-\lfloor x \rfloor) f(\lfloor x \rfloor + 1).
    \end{equation}
    Then, $f\in X$ if and only if $\{f(j)\}_{j\in\Z} \in \ell_X$ and, in this case, $\|f\|_X \approx \|\{f_j\}_j\|_{\ell_X}.$
\end{lemma}

\begin{proof}
    It is immediate that $|f(x)| \leq |f(\lfloor x \rfloor)| + |f(\lfloor x \rfloor + 1)|. $ Since translations are bounded in r.i.\ spaces, this pointwise estimate implies that $f \in X$ if $\{f(j)\}_j \in \ell_X$.
    we obtain that $\{f(j)\}_j \in \ell_X$ implies $f\in X$.

    To get the other implication, denote $r(x) = x - \lfloor x \rfloor$ and consider the sequence $Af$: its $j$-th term is
    \begin{align*}
        (Af)_j &= \int_{j-1}^j f(x) \, dx = \int_{j-1}^j ((1-r(x)) f(\lfloor x \rfloor) + r(x) f(\lfloor x \rfloor + 1)) \, dx \\
        &= \int_{j-1}^j (1-r(x)) f(j-1) \, dx + \int_{j-1}^j r(x) f(j) \, dx \\
        &=  f(j-1) \int_{0}^1 (1-x) \, dx + f(j) \int_{0}^1 x \, dx =  \frac{f(j-1) + f(j)}{2}.
    \end{align*}
    To isolate $f(j)$, we consider the function $r\cdot f$. Note that $|r\cdot f| \leq |f|$ and, if we compute the image through $A$, it results in
    \begin{align*}
        (A(rf))_j &= \int_{j-1}^j r(x) f(x) \, dx = \int_{j-1}^j (r(x)(1-r(x)) f(\lfloor x \rfloor) + r(x)^2 f(\lfloor x \rfloor + 1)) \, dx \\
        &= \int_{j-1}^j r(x)(1-r(x)) f(j-1) \, dx + \int_{j-1}^j r(x)^2 f(j) \, dx \\
        &=  f(j-1) \int_{0}^1 x(1-x) \, dx + f(j) \int_{0}^1 x^2 \, dx =  \frac{1}{6}f(j-1) + \frac{1}{3} f(j).
    \end{align*}

    With this, we can put $f(j) = 6(A(r\cdot f))_j - 2(Af)_j.$ Assuming that $f\in X$, we get that
    $$\|\{f(j)\}_{j\in\Z}\|_{\ell_X} \leq 6 \|A(r\cdot f)\|_{\ell_X} + 2\|Af\|_{\ell_X} \leq 6 \|r\cdot f\|_X + 2\|f\|_X \leq 8\|f\|_X.$$
\end{proof}

With the previous lemma, we can obtain:

\begin{lemma} 
    Let $X$ and $Y$ be r.i.\ spaces. Then, the following are equivalent:
    \begin{enumerate}[label = (\roman*)]
        \item $\ell_X \subseteq \ell_Y$.
        \item \label{ellxellyb} $B_s: X \to Y$ is bounded uniformly on $s > 1$.
        \item \label{ellxellyc} $A_{-1}: X \to Y$ is bounded, where $A_{-1}$ is the operator in (\ref{movingavg}).
    \end{enumerate}
\end{lemma}

\begin{proof}
    Observe that $|A_{-1}f| \leq A_{-1}|f| \leq 2B_1|f| $, so it is clear that \ref{ellxellyb} implies \ref{ellxellyc}.

    For $x\in [m-1,m)$, let us compute $A_{-1} Ba(x)$:
    \begin{align*}
        A_{-1} Ba(x) &= \int_{x-1}^{x} \sum_{j\in\Z} a_j \chi_{(j-1,j)}(y) \, dy = \int_{x-1}^{m-1} a_{m-1} \, dy + \int_{m-1}^{x} a_{m} \, dy \\
        &= a_{m-1} (m-x) + a_m (x-(m-1)) \\
        &= A_{-1} Ba(\lfloor x \rfloor) (1-(x-\lfloor x \rfloor)) + A_{-1} Ba(\lfloor x \rfloor + 1) (x-\lfloor x\rfloor).
    \end{align*}
    Thus, $A_{-1} Ba$ satisfies (\ref{interplin}). Since $A_{-1} Ba \in Y$, we obtain that $\{A_{-1} Ba(j)\}_{j\in\Z} \in \ell_Y$, but $A_{-1} Ba(j) = a_j$. Hence, $\ell_X \subseteq \ell_Y$.

    Assume now that $\ell_X \subseteq \ell_Y$, which is equivalent to $X\cap L^\infty \subseteq Y\cap L^\infty$ by Corollary \ref{cor:ellxlinfty}. Recall that $B_s: X \to X\cap L^\infty$ uniformly on $s>1$. Then, for any $f\in X$ we have that
    \begin{align*}
        \|B_s f\|_Y &\leq \|B_s f\|_{Y\cap L^\infty} \lesssim \|B_s f\|_{X\cap L^\infty} \lesssim \|f\|_X.
    \end{align*}
    That is, $B_s: X\to Y$ is bounded uniformly on $s>1$.
\end{proof}

Now, we are going to establish that the boundedness of any of the three operators implies the uniform boundedness of $\{B_s\}_{s > 1}$.

\begin{lemma}
    Let $X$ and $Y$ be r.i.\ spaces. Assume one of the following:
    \begin{enumerate}[label=(\alph*)]
        \item $S: X \to Y$ is bounded,
        \item $H_1: X \to Y$ is bounded, or
        \item $\H: \ell_X \to \ell_Y$ is bounded.
    \end{enumerate}
    Then, $\ell_X \subseteq \ell_Y$.
\end{lemma}

\begin{proof}
    Suppose that $\H: \ell_X \to \ell_Y$ is bounded. Then, by \cite[Theorem 3]{andersen}, $C: \ell_X \to \ell_Y$ is bounded, where
    $$ (Ca)_n = \frac{1}{n} \sum_{j=1}^n a^*_j, \quad n \geq 1, $$
    and $(Ca)_n = 0$ for $n \leq 0$. Since $a^* \leq Ca$, we have
    $$ \|a\|_{\ell_Y} \leq \|Ca\|_{\ell_Y} \lesssim \|a\|_{\ell_X}, $$
    and the inclusion follows.

    The other two implications are proved by adapting the proof of \cite[Proposition III.4.10]{bensharp}, and using Corollary \ref{cor:ellxlinfty}. Suppose that $H_1:X \to Y$ is bounded, and let $f\in X \cap L^\infty$. Define $g(x) = f^*(-x)$ for $x \leq 0$ and $g(x) = 0$ otherwise. Then, for $x \geq 1$, we have
    $$ H_1 g(x) = \int_0^\infty \frac{f^*(y)}{x+y} \, dy \geq \int_0^x \frac{f^*(y)}{x+y} \, dy \gtrsim \frac{1}{x} \int_0^x f^*(y)\, dy \geq f^*(x). $$
    Then,
    \begin{align*}
        \|f\|_{Y\cap L^\infty} &= \|f^*\|_{Y\cap L^\infty} \leq \|f^* \chi_{(0,1)}\|_{Y\cap L^\infty} + \|f^* \chi_{(1,\infty)}\|_{Y\cap L^\infty} \lesssim \|f^*\|_{L^\infty} + \|H_1 g\|_{Y} \\
        &\lesssim \|f^*\|_{L^\infty} + \|g\|_{X} = \|f^*\|_{L^\infty} + \|f\|_{X} \approx \|f\|_{X \cap L^\infty}.
    \end{align*}

    Finally, suppose that $S: X \to Y$ is bounded, and let $D = \bigcup_{j \geq 0}(2\pi j + [\frac{\pi}{6}, \frac{\pi}{3}))$. For $f\in X\cap L^\infty$, define $h_1(x) = f^*(-x)$ for $x \in -D$ and $h_1(x) = 0$ elsewhere. Then, for $x \in 2\pi m + [\frac{\pi}{6}, \frac{\pi}{3})$,
    \begin{align*}
        Sh_1(x) &= \int_0^\infty \frac{f^*(y) \chi_D(y) \sin(x+y)}{x+y} \, dy \gtrsim \int_0^x \frac{f^*(y) \chi_D(y)}{x+y} \, dy \gtrsim \frac{1}{x} \int_0^x f^*(y) \chi_D(y) \, dy \\
        &\gtrsim \frac{1}{x} \int_0^{2\pi (m-1) + \frac{\pi}{3}} f^*(y) \chi_D(y) \, dy.
    \end{align*}
    Now, define the function
    $$ h_2(y) = \frac{1}{12} f^*\left( \frac{y + 22\pi j}{12} + \frac{\pi}{6} \right), \qquad y \in 2\pi j + [0,2\pi). $$
    This function is decreasing in $[0,\infty)$ and satisfies
    $$ \int_0^{2\pi m} h_2(y) \, dy = \int_0^{2\pi (m-1) + \frac{\pi}{3}} f^*(y) \chi_D(y) \, dy, $$
    so we obtain that
    $$ Sh_1(x) \gtrsim \frac{1}{x} \int_0^{2\pi m} h_2(y) \, dy \gtrsim \frac{1}{2\pi m} \int_0^{2\pi m} h_2(y) \, dy \geq h_2(2\pi m) \approx f^*\left(2\pi m + \frac{\pi}{6}\right). $$
    In other words, denoting $D_1 = D \cap [2\pi, \infty)$, we have arrived at the inequality
    $$ Sh_1(x) \chi_{D_1}(x) \gtrsim \sum_{m=1}^\infty f^*\left(2\pi m + \frac{\pi}{6}\right) \chi_{2\pi m + [\frac{\pi}{6},\frac{\pi}{3})}(x) =: \Tilde{f}(x). $$
    To finish the proof, we observe that
    $$ \|Sh_1\|_Y \gtrsim \|(\Tilde{f})^*\|_Y = \Big\|E_{\frac{6}{\pi}}B\Big(\left\{f^*\left(2\pi m + \frac{\pi}{6}\right)\right\}_{m\geq 1}\Big)\Big\|_Y \approx \|f^* \chi_{(2\pi + \frac{\pi}{6}, \infty)} \|_{Y\cap L^\infty}, $$
    so, since $f^* \chi_{(0, 2\pi + \frac{\pi}{6})} \in Y\cap L^\infty$, it follows that $f \in Y \cap L^\infty$.
\end{proof}

\begin{rem} 
    Because translations are isometries on r.i.\ spaces, Lemma~\ref{convl1} holds and we have checked that $B_s: X \to Y$ is uniformly bounded for $s>1$ in any of the three cases, the hypotheses of Theorem \ref{charbfs} are satisfied. Hence, Theorem~\ref{charbfs} gives the equivalence between \ref{mainria} and \ref{mainrib}.
\end{rem}

The remainder of this proof will be focused on establishing that \ref{mainrib} is equivalent to \ref{mainric}. To show this, we recall another operator that will be useful.

\begin{defo}
    We denote as $K$ the convolution operator with kernel $\frac{1}{\lfloor x \rfloor} \chi_{(-1,1)^c}. $
\end{defo}

Note that this operator (modulo some dilation) was already introduced in (\ref{defk}), in the proof of Theorem~\ref{charbfs}.

\begin{lemma} \label{reduclba}
    Let $X$ and $Y$ be r.i.\ spaces such that $B_s: X \to Y$ is bounded uniformly on $s > 1$, and recall the operators $A$ and $B$ introduced in (\ref{ba}). Then, the following are equivalent:
    \begin{enumerate}[label=(\roman*)]
        \item\label{reduclbai} $H_1: X \to Y$ is bounded,
        \item\label{reduclbaii} $K: X \to Y$ is bounded,
        \item $KBA: X \to Y$ is bounded.
    \end{enumerate}
\end{lemma}

\begin{proof}
    Recalling inequality (\ref{eq:difker}) for $|x|>1$, since we know that the operator $P$ is bounded from $X$ to $Y$, the equivalence between \ref{reduclbai} and \ref{reduclbaii} holds.

    To get the equivalence between $K$ and $KBA$, we aim to study their difference, $KBA - K$. First, note that if $f\in X$ and $n\in \Z$, then we have that
    \begin{equation}\label{Kinteger}
    \begin{split}
        Kf(n) &= \int_\R \frac{1}{\lfloor y \rfloor} \chi_{(-1,1)^c}(y) f(n-y) \, dy = \int_\R \sum_{j\neq 0,-1} \frac{1}{j} \chi_{(j, j+1)}(y) f(n-y) \, dy \\
        &= \sum_{j\neq 0,-1} \frac{1}{j} \int_j^{j+1} f(n-y) \, dy = \sum_{j\neq 0,-1} \frac{1}{j} \int_{n-j-1}^{n-j} f(y) \, dy \\
        &= \sum_{j\neq 0,-1} \frac{1}{j} (Af)_{n-j} = \sum_{j\neq 0,-1} \frac{1}{j} \int_{n-j-1}^{n-j} BAf(y) \, dy \\
        &= KBAf(n).
    \end{split}
    \end{equation}

    Now, let us compare evaluations of $Kf$ at two close points. Considering $x \in \R$ with $|x| \geq 2$ and $0\leq h \leq 1$, we have
    $$ |k(x+h) - k(x)| = \frac{1}{\lfloor x \rfloor} - \frac{1}{\lfloor x + h \rfloor} = \frac{\lfloor x+h \rfloor - \lfloor x \rfloor}{\lfloor x \rfloor \lfloor x+h \rfloor} \leq \frac{1}{\lfloor x \rfloor \lfloor x+h \rfloor} \approx \frac{1}{x^2 + 1}. $$
    On the other hand, since $k\in L^\infty$, if $|x| \leq 2$ we have that
    $$ |k(x+h) - k(x)| \lesssim \frac{1}{x^2 + 1}. $$
    Therefore, we obtain that
    $$ |Kf(x+h) - Kf(x)| \leq (|k(y+h) - k(y)|*|f|)(x) \lesssim P(|f|)(x). $$

    With this estimate, we can move on to bound the difference $KBAf - Kf$. Recalling that $KBAf(n) = Kf(n)$ (see (\ref{Kinteger})) and noting that $\lfloor x\rfloor +1 - x \in (0,1]$, we now have
    \begin{align*}
        |KBAf(x) - Kf(x)| &= |K(BAf - f)(x)| \\
        &= |K(BAf - f)(x) - K(BAf - f)(\lfloor x\rfloor +1)| \\
        &\lesssim P(|BAf - f|)(x) = P(|(BA-I)f|)(x) .
    \end{align*}
    Now, observe that all of the operators involved on the right-hand side are bounded. Hence, we conclude that the difference operator $KBA - K$ is always bounded from $X$ to $Y$. From this, it is clear that $K$ is bounded if and only if $KBA$ is.
\end{proof}

The following steps aim to connect the boundedness of $KBA$ to that of the discrete Hilbert transform $\H$.

\begin{lemma} \label{Kfsep}
    Let $X$ be an r.i.\ space, and let $f\in X$ be in the image of the operator $B$; that is, $f = \sum_{j\in\Z} f_j \chi_{(j-1,j)}. $ Then, the function $Kf$ satisfies
    \begin{equation} \label{Kinterplin}
        Kf(x) = (x-\lfloor x \rfloor) Kf(\lfloor x \rfloor + 1) + (1- (x-\lfloor x\rfloor)) Kf(\lfloor x \rfloor).
    \end{equation}
\end{lemma}

\begin{proof}
    Let us find an alternative expression for the operator $K$. If $g\in X$ is arbitrary, we have
    \begin{align*}
        Kg(x) &= \int_\R \frac{1}{\lfloor y \rfloor} \chi_{(-1,1)^c}(y) g(x-y) \, dy = \int_\R \sum_{j\neq 0,-1} \frac{1}{j} \chi_{(j, j+1)}(y) g(x-y) \, dy \\
        &= \sum_{j\neq 0,-1} \frac{1}{j} \int_j^{j+1} g(x-y) \, dy = \sum_{j\neq 0,-1} \frac{1}{j} \int_{x-j-1}^{x-j} g(y) \, dy \\
        &= \sum_{j\neq 0,-1} \frac{1}{j} \left(\int_{\lfloor x\rfloor -j}^{x-j} g(y) \, dy + \int_{x-j-1}^{\lfloor x\rfloor -j} g(y) \, dy\right).
    \end{align*}

    Now, if we plug in the function $f$, we get
    \begin{align*}
        Kf(x) &= \sum_{j\neq 0,-1} \frac{1}{j} \left(\int_{\lfloor x\rfloor -j}^{x-j} f(y) \, dy + \int_{x-j-1}^{\lfloor x\rfloor -j} f(y) \, dy\right) \\
        &= \sum_{j\neq 0,-1} \frac{1}{j} (f_{\lfloor x\rfloor -j} \cdot (x-\lfloor x \rfloor) + f_{\lfloor x\rfloor -j - 1} \cdot (1- (x-\lfloor x \rfloor)).
    \end{align*}

    In particular, for $n\in \Z$, evaluating $Kf$ at $n$ yields
    $$ Kf(n) = \sum_{j\neq 0,-1} \frac{f_{n -j - 1}}{j}. $$
    With this, we can write
    \begin{align*}
        Kf(x) &= (x-\lfloor x \rfloor) \sum_{j\neq 0,-1} \frac{1}{j} f_{\lfloor x\rfloor -j} + (1- (x-\lfloor x \rfloor)) \sum_{j\neq 0,-1} \frac{1}{j} f_{\lfloor x\rfloor -j - 1} \\
        &= (x-\lfloor x \rfloor) Kf(\lfloor x \rfloor + 1) + (1- (x-\lfloor x\rfloor)) Kf(\lfloor x \rfloor),
    \end{align*}
    which is the desired equality.
\end{proof}

\begin{lemma} \label{compKH}
    Let $X$ be an r.i.\ space, and let $f\in X$. Then, the operator $K$ and the discrete Hilbert transform $\H$ satisfy the following equalities:
    $$Kf(k) = (\mathcal{H}Af)_k + (Af)_{k+1},$$
    $$(\mathcal{H}g)_k = KBg(k) - g_{k+1}.$$
\end{lemma}

\begin{proof}
    In the proof to Lemma~\ref{Kfsep} we have seen that
    $$Kf(k) = \sum_{j\neq 0,-1} \frac{1}{j} \int_{k-j-1}^{k-j} f(y)\, dy.$$

    If we compute the $k$-th term of $\mathcal{H}Af$, we get
    \begin{align*}
        (\H A f)_k &= \sum_{j\neq 0} \frac{1}{j} (Af)_{k-j} = \sum_{j\neq 0} \frac{1}{j} \int_{k-j-1}^{k-j} f(y)\, dy.
    \end{align*}

    We see that both series are very similar, with the only difference being that $Kf$ is missing the term for $j = -1$. Adding it, we obtain the first conclusion:
    $$(\H A f)_k = Kf(k) - \int_{k}^{k+1} f(y) \, dy. $$

    The second conclusion follows directly from putting $f = Bg$ in the first equality and recalling that $AB = I_{\ell_X}$.
\end{proof}

Now, we can state the last lemma needed for the equivalence.

\begin{lemma} \label{truncdisc}
    Let $X$ and $Y$ be r.i.\ spaces such that $B_s: X \to Y$ is bounded uniformly on $s>1$. Then, the operator $KBA: X \to Y$ is bounded if and only if the discrete Hilbert transform $\H: \ell_X \to \ell_Y$ is bounded.
\end{lemma}

\begin{proof}
    Suppose that $KBA$ is bounded. Then, because of Lemma~\ref{Kfsep} we know that for any $f\in X$, $KBAf$ can be expressed as in (\ref{Kinterplin}), so Lemma~\ref{lem:interplin} can be used for $KBAf$, giving
    $$\|KBAf\|_{Y} \approx \|\{KBAf(j)\}_j\|_{\ell_Y}.$$
    Now, given any sequence $g\in \ell_X$, we can see that $KBg = KBA(Bg)$. Thus, by Lemma~\ref{compKH},
    $$(\H g)_k = KBg(k) - g_{k+1} = KBA(Bg)(k) - (T_{-1}g)_k. $$
    Taking norms and considering that $\ell_X \subseteq \ell_Y$, it follows that
    \begin{align*}
        \|\H g\|_{\ell_Y} &\leq \|\{KBA(Bg)(k)\}_k\|_{\ell_Y} + \|T_{-1}g\|_{\ell_Y} \approx \|KBA(Bg)\|_Y + \|T_{-1}g\|_{\ell_Y} \\
        &\lesssim \|Bg\|_X + \|g\|_{\ell_X} = 2 \|g\|_{\ell_X},
    \end{align*}
    so $\H$ is bounded from $\ell_X$ to $\ell_Y$.

    Analogously, suppose that $\H$ is bounded. Then, for any $f\in X$, Lemma~\ref{compKH} gives
    \begin{align*}
        \|\{Kf(k)\}_k\|_{\ell_Y} &\leq \|\H Af\|_{\ell_Y} + \|T_{-1}Af\|_{\ell_Y} \lesssim \|Af\|_{\ell_X} \lesssim \|f\|_X.
    \end{align*}
    Therefore, putting $BAf$ in place of $f$, we have that
    \begin{align*}
        \|KBAf\|_Y \approx \|\{KBAf(k)\}_k\|_{\ell_Y} &\lesssim \|BAf\|_X \lesssim \|f\|_X.
    \end{align*}
    Hence, $KBA$ is bounded from $X$ to $Y$.
\end{proof}

Joining together Lemma~\ref{reduclba} and Lemma~\ref{truncdisc}, we can see that the boundedness of $H_1$ is equivalent to $KBA$ which is, in turn, equivalent to $\H$. Thus, the equivalence between \ref{mainrib} and \ref{mainric} is proved, and the proof of Theorem~\ref{mainri} is complete.
\end{proof}

\begin{rem} \label{rem:mainri}
    Observe that $S: X \to Y$ is bounded if and only if $S: Z \to W$ is, where $Z$ and $W$ are any r.i.\ spaces with $\ell_Z = \ell_X$ and $\ell_W = \ell_Y$, because both statements are equivalent to the boundedness of $\H: \ell_X \to \ell_Y$.

    In particular, considering Proposition \ref{prop:sumint}, we can find that the following are equivalent:
    \begin{enumerate}
        \item $S: X \to Y$ is bounded,
        \item $S: X + L^1 \to Y \cap L^\infty$ is, and
        \item $S: X \cap L^\infty \to Y + L^1$ is.
    \end{enumerate}
    
    The same can be said for $H_1$.
\end{rem}

Theorem \ref{mainri} gives an equivalence in the non-diagonal case of the boundedness of three different operators. Furthermore, when we look into the diagonal case when $X = Y$, we can use the previously known characterization for $\H$ by the Boyd indices of the space to extend it to the following result:

\begin{thm} \label{cor:mainri}
    Let $X(\R)$ be an r.i.\ space, and let $\underline{d}_X, \overline{d}_X$ be its discrete Boyd indices. Then, the following are equivalent:
    \begin{enumerate}[label=(\alph*)]
        \item \label{cor:mainria} $S$ is bounded on $X$,
        \item $S: X + L^1 \to X \cap L^\infty$ is bounded,
        \item $H_1$ is bounded on $X$,
        \item $H_1: X + L^1 \to X \cap L^\infty$ is bounded,
        \item \label{cor:mainrie} $\H$ is bounded on $\ell_X$,
        \item \label{cor:mainrif} $0 < \underline{d}_X \leq \overline{d}_X < 1.$
    \end{enumerate}
\end{thm}

\begin{proof}
    The equivalences \ref{cor:mainria}-\ref{cor:mainrie} are Theorem \ref{mainri} for $X = Y$ and Remark \ref{rem:mainri}. The equivalence between \ref{cor:mainrie} and \ref{cor:mainrif} is given by the result~\cite[(4')]{andersen}.
\end{proof}

To end this section, we will use this last theorem to find r.i.\ spaces where the segment multiplier is bounded but the Hilbert transform is not. This stands in contrast to the simultaneous boundedness in $L^p$, showing that in the r.i.\ setting the behavior of $S$ and $H$ is quite different. To show that our space actually satisfies these properties, we make use of the indices computed in Example \ref{exdiscrind}.

\begin{ex} \label{ex:l1+lp}
    Consider the space $(L^1 + L^p)(\R)$, for some $1 < p < \infty$. Then, we know that $\overline{\alpha}_{L^1 + L^p} = 1,$ so the Hilbert transform is not bounded on $L^1 + L^p$. On the other hand, as seen in Example \ref{exdiscrind},
    $$ 0 < \underline{d}_{L^1 + L^p} = \frac{1}{p} = \overline{d}_{L^1 + L^p} < 1, $$
    and, by Theorem~\ref{mainri}, the segment multiplier is bounded on $L^1 + L^p$. By duality, the same result holds for the intersection space $(L^p \cap L^\infty)(\R)$.
\end{ex}

\end{document}